\documentclass{amsart}
\usepackage{graphicx}
\newtheorem{theorem}{Theorem}[section]
\newtheorem{lemma}[theorem]{Lemma}
\newtheorem{corollary}[theorem]{Corollary}

\newtheorem{problem}[theorem]{Problem}

\theoremstyle{definition}

\newtheorem{example}[theorem]{Example}

\newtheorem{claim}[theorem]{Claim}
\newtheorem{conjecture}[theorem]{Conjecture}

\theoremstyle{remark}

\numberwithin{equation}{section}

\begin{document}

\title{Bridge position and the representativity of spatial graphs}

\author{Makoto Ozawa}
\address{Department of Natural Sciences, Faculty of Arts and Sciences, Komazawa University, 1-23-1 Komazawa, Setagaya-ku, Tokyo, 154-8525, Japan}
\curraddr{Department of Mathematics and Statistics, The University of Melbourne
Parkville, Victoria 3010, Australia (until March 2011)}
\email{w3c@komazawa-u.ac.jp, ozawam@unimelb.edu.au (until March 2011)}

\subjclass[2000]{Primary 57M25; Secondary 57Q35}



\keywords{knot, spatial graph, bridge position, bridge number, representativity}

\begin{abstract}
First, we extend Otal's result for the trivial knot to trivial spatial graphs, namely, we show that for any bridge tangle decomposing sphere $S^2$ for a trivial spatial graph $\Gamma$, there exists a 2-sphere $F$ such that $F$ contains $\Gamma$ and $F$ intersects $S^2$ in a single loop.

Next, we introduce two invariants for spatial graphs.
As a generalization of the bridge number for knots, we define the {\em bridge string number} $bs(\Gamma)$ of a spatial graph $\Gamma$ as the minimal number of $|\Gamma\cap S^2|$ for all bridge tangle decomposing sphere $S^2$.
As a spatial version of the representativity for a graph embedded in a surface, we define the {\em representativity} of a non-trivial spatial graph $\Gamma$ as
\[
r(\Gamma)=\max_{F\in\mathcal{F}} \min_{D\in\mathcal{D}_F} |\partial D\cap \Gamma|,
\]
where $\mathcal{F}$ is the set of all closed surfaces containing $\Gamma$ and $\mathcal{D}_F$ is the set of all compressing disks for $F$ in $S^3$.
Then we show that for a non-trivial spatial graph $\Gamma$,
\[
\displaystyle r(\Gamma)\le \frac{bs(\Gamma)}{2}.
\]
In particular, if $\Gamma$ is a knot, then $r(\Gamma)\le b(\Gamma)$, where $b(\Gamma)$ denotes the bridge number.
This generalizes Schubert's result on torus knots.

\end{abstract}

\maketitle
\tableofcontents

\section{Introduction}

Knots, links and spatial graphs are $1$-dimensional objects which are embedded in the $3$-dimensional space.
$2$-dimensional objects between them, namely surfaces, frequently extract some useful information.
In this paper, we concern with closed surfaces which contains these $1$-dimensional objects.

In the following subsections \ref{sub1} and \ref{sub2}, we study bridge positions of trivial and non-trivial spatial graphs respectively.
In subsection \ref{sub3}, we state some results on the representativity of knots.
We will show the existence of spatial graphs with high representativity in subsection \ref{sub4}, and obtain some results on them in subsection \ref{sub5}.
Finally, we propose the strong spatial embedding conjecture in subsection \ref{sub6}.

\subsection{Bridge position of trivial spatial graphs}\label{sub1}

In \cite{O}, Otal proved that any two $n$-bridge presentations of the trivial knot are isotopic.
This shows that if the trivial knot $K$ has a trivial tangle decomposition $(S^3,K)=(B_+,K_+)\cup_{S^2}(B_-,K_-)$, then there exists a 2-sphere $F$ such that $F$ contains $K$ and $F$ intersects $S^2$ in a single loop.
Thus $K_+$ and $K_-$ can be isotoped into $S^2$ in $B_+$ and $B_-$ respectively so that $\text{int}K_+\cap \text{int}K_-=\emptyset$, and hence we have a trivial diagram of $K$ on $S^2$ with fixing $K\cap S^2$.
It is worth to notice that an isotopy of unknotting $K$ can be decomposed into two isotopies of unknotting $K_+$ and $K_-$ by any bridge tangle decomposing sphere $S^2$.
In Theorem \ref{main}, we extend this phenomenon to the case of trivial spatial graphs.

Let $G$ be a graph and $f:G\to S^3$ be an embedding of $G$ into $S^3$.
We call the image $\Gamma=f(G)$ a {\em spatial graph} of $G$.
A spatial graph $\Gamma$ is {\em trivial} if there exists a 2-sphere which contains $\Gamma$.
Let $S^2$ be a 2-sphere which divides $S^3$ into two 3-balls $B_+$ and $B_-$, and suppose that $\Gamma$ intersects $S^2$ transversely in the interior of edges.
Put $\Gamma_{\pm}=\Gamma\cap B_{\pm}$.
Then we say that $(B_+,\Gamma_+)\cup_{S^2}(B_-,\Gamma_-)$ is a {\em bridge tangle decomposition} of a pair $(S^3,\Gamma)$ and $S^2$ is a {\em bridge tangle decomposing sphere} for $\Gamma$ if there exists a disk $D_{\pm}$ in $B_{\pm}$ containing $\Gamma_{\pm}$ and $\Gamma_{\pm}$ consists of trees or arcs.


We remark that if a pair $(S^3,\Gamma)$ has a bridge tangle decomposition $(B_+,\Gamma_+)\cup_{S^2}(B_-,\Gamma_-)$ with a bridge tangle decomposing sphere $S^2$, then it has a {\em bridge position} with respect to the standard height function $h:S^3\to \Bbb{R}$, namely, $h^{-1}([0,\pm\infty))=B_{\pm}$, $h^{-1}(\{0\})=S^2$ and $\Gamma_{\pm}$ consists of ``monotone trees" (each edge has no critical point) or arcs with only one maximal or minimal point with respect to $h$.
Conversely, it follows that if a spatial graph $\Gamma$ has a bridge position, then it has a corresponding bridge tangle decomposition.

\begin{theorem}\label{main}
For any bridge tangle decomposing sphere $S^2$ for a trivial spatial graph $\Gamma$, there exists a 2-sphere $F$ such that $F$ contains $\Gamma$ and $F$ intersects $S^2$ in a single loop.
\end{theorem}


\subsection{Bridge position and the representativity of non-trivial spatial graphs}\label{sub2}

We define the {\em bridge string number} $bs(\Gamma)$ of a spatial graph $\Gamma$ as the minimal number of $|\Gamma\cap S^2|$ for all bridge tangle decompositions $(B_+,\Gamma_+)\cup_{S^2}(B_-,\Gamma_-)$, namely,
\[
bs(\Gamma)=\min_{\Gamma\in \mathcal{BP}_{\Gamma}} |\Gamma\cap S|,
\]
where $\mathcal{BP}_{\Gamma}$ is the set of all bridge position of $\Gamma$.
When $\Gamma$ is a knot, the bridge string number $bs(K)$ is equal to the twice of the bridge number $b(K)$ introduced in \cite{S}.
For spatial graphs of a $\theta$-curve graph, the bridge number $b(\Gamma)$ is also defined in \cite{G} or \cite{M} and $bs(\Gamma)\le 2b(\Gamma)+1$ holds for a spatial graph $\Gamma$ of a $\theta$-curve.

Let $\Gamma$ be a non-trivial spatial graph and $F$ a closed surface containing $\Gamma$.
Following \cite{MO1}, the {\em representativity} $r(F,\Gamma)$ of a pair $(F,\Gamma)$ is defined as the minimal number of intersecting points of $\Gamma$ and $\partial D$, where $D$ ranges over all compressing disks for $F$ in $S^3$, namely,
\[
\displaystyle r(F,\Gamma)=\min_{D\in\mathcal{D}_F} |\partial D\cap \Gamma|,
\]
where $\mathcal{D}_F$ is the set of all compressing disks for $F$ in $S^3$.
Here, $\partial D$ may run over some vertices of $\Gamma$.

Furthermore, we define the {\em representativity} $r(\Gamma)$ of a non-trivial spatial graph $\Gamma$ as the maximal number of $r(F,\Gamma)$ over all closed surfaces $F$ containing $\Gamma$, namely,
\[
r(\Gamma)=\max_{F\in\mathcal{F}} r(F,\Gamma),
\]
where $\mathcal{F}$ is the set of all closed surfaces of positive genus containing $\Gamma$.
We note that $r(\Gamma)\ge 1$ for any non-trivial spatial graph $\Gamma$.
The representativity of a knot measures how many times the knot can be wrapped around a closed surface, and the representativity of a spatial graph can be considered as a spatial version of the representativity for a graph embedded in a surface (\cite{RV}).



The following is a main theorem which states the title of this paper.

\begin{theorem}\label{main2}
For a non-trivial spatial graph $\Gamma$,
\[
\displaystyle r(\Gamma)\le \frac{bs(\Gamma)}{2}.
\]
\end{theorem}

Theorem \ref{main2} states the representativity takes a finite value, and as in Section \ref{determination}, the representativity can be used to give a lower bound for the bridge string number.

\subsection{The representativity of knots}\label{sub3}
We summarize what are known about the representativity of knots.

\begin{theorem}[\cite{MO4}]\label{knots}
We have 
\begin{enumerate}
\item $2\le r(K)\le b(K)$ for a non-trivial knot $K$.
\item $r(K)=\min \{p, q\}$ for a $(p,q)$-torus knot $K$ $(p,q>0)$.
\item $r(K)=2$ for a $2$-bridge knot $K$.
\item $r(K)\le 3$ for an algebraic knot $K$.
\item $r(K)=3$ for a $(p,q,r)$-pretzel knot $K$ if and only if $(p,q,r)=\pm (-2,3,3)$ or $\pm(-2,3,5)$.
\end{enumerate}
\end{theorem}


\begin{proof}
(1) We have $r(F,K)=2$ for a closed surface $F=\partial N(S)$, where $S$ is a minimal genus Seifert surface for $K$, and hence we have $r(F,K)=2 \le r(K)$.
The latter inequality $r(K)\le b(K)$ follows from Theorem \ref{main2}.

(2) When $K$ is a torus knot of type $(p,q)$ and $F$ is an unknotted torus containing $K$, it follows from Theorem \ref{main2} that $\min\{p,q\}=r(F,K)\le r(K)\le b(K)$.
We remark that this gives an equality $b(K)=\min\{p,q\}$ as proved in \cite{S} and \cite{Sch}.

(3) For any 2-bridge knot $K$, it follows from Theorem \ref{main2} that $2\le r(K)\le b(K)=2$.
Hence $r(K)=2$.

(4) and (5) are main theorems proved in \cite{MO4}.
\end{proof}


\begin{conjecture}
It holds $r(K)=2$ for an alternating knot $K$.
\end{conjecture}

\begin{problem}
Characterize knots having $r(K)=b(K)$.
\end{problem}



\subsection{Constructing spatial graphs with arbitrarily high representativity}\label{sub4}

Let $g(G)$ denote the minimal genus of a graph $G$, that is, the minimal genus of orientable closed surfaces containning $G$.
The following theorem says that there exist spatial graphs with arbitrarily high representativity.

\begin{theorem}\label{arbitrarily}
For any closed surface $F$ embedded in $S^3$ with $g(F)\ge g(G)\ge 1$ and for any integer $n$, there exists a spatial graph $\Gamma$ of $G$ contained in $F$ such that $r(F,\Gamma)\ge n$.
\end{theorem}






\subsection{Some properties derived from the representativity}\label{sub5}
In this subsection, we derive next two properties of spatial graphs from the representativity.

We say that a spatial graph $\Gamma$ is {\em totally knotted} if $\partial (S^3-\text{int}N(\Gamma))$ is incompressible in $S^3-\text{int}N(\Gamma)$.

\begin{theorem}\label{totally knotted}
If $r(\Gamma)>\beta_1(G)$, then $\Gamma$ contains a connected totally knotted spatial subgraph, where $\beta_1(G)$ denotes the first Betti number of $G$.
\end{theorem}

Let $S$ be a 2-sphere which divides $S^3$ into two 3-balls $B_1$ and $B_2$, and suppose that a spatial graph $\Gamma$ intersects $S$ transversely in the interior of edges.
Put $\Gamma_i=\Gamma\cap B_i$.
Then we say that $(B_1,\Gamma_1)\cup_S(B_2,\Gamma_2)$ is an {\em essential tangle decomposition} of a pair $(S^3,\Gamma)$  if $S-\Gamma$ is incompressible in $S^3-\Gamma$ and there exists no disk $D_i$ properly embedded in $B_i$ which contains $\Gamma_i$.
Here we call $S$ an {\em essential tangle decomposing sphere} for $\Gamma$.
We say that a spatial graph $\Gamma$ is {\em spatially $n$-connected} if it has no essential tangle decomposing sphere $S$ with $|\Gamma\cap S|<n$.

\begin{theorem}\label{spatially}
If $r(\Gamma)=n$, then $\Gamma$ is spatially $n$-connected.
\end{theorem}

\begin{corollary}
We have 
\begin{enumerate}
\item $r(K)=2$ for any composite knot $K$.
\item $r(K)\le 4$ for a knot $K$ with an essential Conway sphere.
\end{enumerate}
\end{corollary}

\subsection{Strong spatial embedding conjecture}\label{sub6}

Let $G$ be a non-planar graph which is contained in a closed surface $F$.
Robertson and Vitray \cite{RV} defined the {\em representativity} of a pair $(F,G)$ as
\[
r(F,G)=\min_{C\in \mathcal{C}_F}|C\cap G|,
\]
where $\mathcal{C}_F$ is the set of all essential loops embedded in $F$.
Moreover we define the {\em representativity} of $G$ as
\[
r(G)=\max_{F\in \mathcal{F}} r(F,G),
\]
where $\mathcal{F}$ is the set of all closed surfaces containing $G$.

Then the strong embedding conjecture can be stated as follows.

\begin{conjecture}[Strong embedding conjecture \cite{J}]
It holds $r(G)\ge 2$ for a $2$-connected non-planar graph $G$.
\end{conjecture}

Similary we propose the following conjecture.

\begin{conjecture}[Strong spatial embedding conjecture]
It holds $r(\Gamma)\ge 2$ for a non-trivial spatial graph $\Gamma$ of a $2$-connected graph $G$.
\end{conjecture}

\section{Preparations}

\subsection{Essential Morse bridge position}

Let $\Gamma$ be a spatial graph of a graph $G$ with a bridge tangle decomposition $(B_+,\Gamma_+)\cup_{S}(B_-,\Gamma_-)$, and $F$ be a closed surface which contains $\Gamma$.

We define a {\em Morse bridge position} of a pair $(F,\Gamma)$ with respect to a Morse function $h:S^3\to \Bbb{R}$ with two critical points as follows.
We divide $S^3$ into three parts $M_+$, $M_0$ and $M_-$ so that $M_{\pm}\cap M_0=h^{-1}(c_{\pm})$ is a level 2-sphere for $c_+>0>c_-$.
We say that a pair $(F,\Gamma)$ with a bridge tangle decomposition $(B_+,\Gamma_+)\cup_{S}(B_-,\Gamma_-)$ is in a {\em Morse bridge position} if 
\begin{enumerate}
\item $h|_F$ is a Morse function,
\item $\Gamma$ is disjoint from critical points of $F$,
\item all maximal (resp. minimal) points of $F$ are contained in $M_+$ (resp. $M_-$),
\item all saddle points of $F$ are contained in $M_0$,
\item $(M_{\pm},\Gamma\cap M_{\pm})$ is homeomorphic to $(B_{\pm},\Gamma_{\pm})$ as a pair,
\item $(M_0,\Gamma\cap M_0)$ is homeomorphic to $(S,\Gamma\cap S)\times I$ as a pair,
\item $S=h^{-1}(0)$.
\end{enumerate}
See Figure \ref{morse}.

\begin{figure}[htbp]
	\begin{center}
	\includegraphics[trim=0mm 0mm 0mm 0mm, width=.5\linewidth]{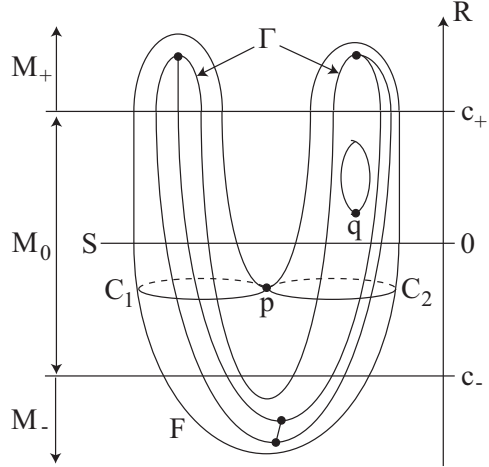}
	\end{center}
	\caption{$(F,\Gamma)$ in a Morse bridge position}
	\label{morse}
\end{figure}


\begin{lemma}\label{lemma1}
A pair $(F,\Gamma)$ with a bridge tangle decomposition $(B_+,\Gamma_+)\cup_{S}(B_-,\Gamma_-)$ can be isotoped so that it is in a Morse bridge position.
\end{lemma}

\begin{proof}
We basically follow the argument of \cite[Claim 4.1]{MO}.

Since $(B_{\pm},\Gamma_{\pm})$ is a trivial bridge tangle, there exists a disk $D_{\pm}$ in $B_{\pm}$ which contains $\Gamma_{\pm}$ such that there exists a region $R_{\pm}$ of $D_{\pm}-\Gamma_{\pm}$ which contacts with all components of $\Gamma_{\pm}$.
We take the spanning forest $T_{\pm}$ (a ``middle'' point when a component of $T_{\pm}$ is an arc) for $\Gamma_{\pm}$ and connect each component of $T_{\pm}$ with a point $x_{\pm}$ in $R_{\pm}$ by an arc $\alpha_i$ so that $N(T_{\pm}\cup \bigcup \alpha_i;D_{\pm})$ is a disk (Figure \ref{disk}).
Then we may assume that $F$ intersects each component of $N(T_{\pm})$ in a disk and $N(\alpha_i)$ in disks.
We put $M_{\pm}=N(T_{\pm}\cup \bigcup \alpha_i)$ and $M_0=S^3-\text{int}(M_+\cup M_-)$.
Then the condition (5) holds, and (6) also holds since the 2-sphere $\partial N(T_{\pm}\cup \bigcup \alpha_i)$ is parallel to $\partial B_{\pm}$ in $(B_{\pm},\Gamma_{\pm})$.

Next let $h:S^3\to \Bbb{R}$ be a Morse function with two critical points $x_+$ and $x_-$.
We may assume that $M_{\pm}\cap M_0=h^{-1}(c_{\pm})$ for $c_+>0>c_-$ and the condition (7) $S=h^{-1}(0)$ by (6).
Since $F\cap M_{\pm}$ consists of disks, we may assume that each disk has only one critical point with respect to $h$.
Morse theory provides that every critical point of $F\cap M_0$ is either maximal point, saddle point or minimal point.
Thus the condition (1) holds, and (2) also holds if we isotope $\Gamma$ slightly.
Finally if we pull up or down maximal points or minimal points of $F\cap M_0$ into $M_+$ or $M_-$ respectively, then the conditions (3) and (4) hold.
\end{proof}

\begin{figure}[htbp]
	\begin{center}
	\includegraphics[trim=0mm 0mm 0mm 0mm, width=.4\linewidth]{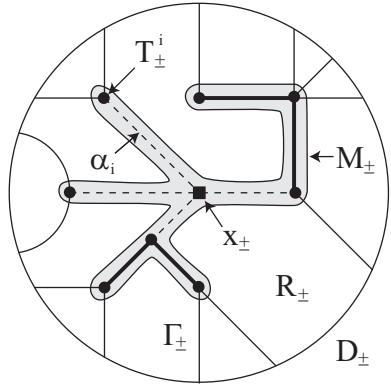}
	\end{center}
	\caption{$M_{\pm}=N(T_{\pm}\cup \bigcup \alpha_i)$}
	\label{disk}
\end{figure}

A saddle point $p$ of $F$ is {\em inessential} if at least one of the two loops $C_1$, $C_2\subset h^{-1}(h(p))$ contacting with $p$ is inessential in $F$.
In Figure \ref{morse}, $p$ is an inessential saddle point and $q$ is an essential saddle point of $F$.

\begin{lemma}\label{lemma2}
A pair $(F,\Gamma)$ with a bridge tangle decomposition $(B_+,\Gamma_+)\cup_{S}(B_-,\Gamma_-)$ can be isotoped so that $F$ has no inessential saddle point in a Morse bridge position.
\end{lemma}

\begin{proof}
We basically follow the argument of \cite{S}, \cite{Sch} or \cite{O}.

Suppose that $F$ has an inessential saddle points, and let $p$ be an innermost inessential saddle point, that is, one of the two loops $C_1,\ C_2$ in $h^{-1}(h(p))$ contacting with $p$ bounds a disk in $F$ which contains only one maximal or minimal point.
Without loss of generality, we may assume that $C_1$ bounds a disk $D$ in $F$ which contains only one maximal point.
By exchanging the critical point $x_+$ in $M_+-(F\cap M_+)$ if necessary, we may assume that $C_1$ bounds a disk $D'$ in $h^{-1}(h(p))$ which does not contain $C_2$.

There are two cases.
\begin{enumerate}
\item [(a)] $h(p)>0$
\item [(b)] $h(p)\le 0$
\end{enumerate}

In case (a), take a monotone arc $\alpha$ connecting $p$ with a point of $(F\cap M_+)-\Gamma_+$ in $F-\Gamma$.
Let $B$ be a 3-ball bounded by a 2-sphere $D\cup D'$ which does not contain $\alpha$.
We first isotope $D$ into the region below $D'$ by shrinking $B$ vertically so that the maximal point of $D$ can be cancelled with $p$.
Next we isotope the shrinked $B$ by sliding along $\alpha$ so that $(F,\Gamma)$ is in a Morse bridge position again.
See Figure \ref{isotopy}.

\begin{figure}[htbp]
	\begin{center}
	\includegraphics[trim=0mm 0mm 0mm 0mm, width=.8\linewidth]{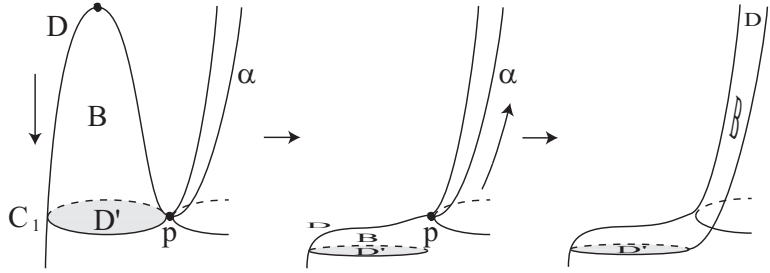}
	\end{center}
	\caption{An isotopy which cancels an inessential saddle point $p$}
	\label{isotopy}
\end{figure}

In case (b), we can isotope $F\cap h^{-1}([h(p)-\epsilon,0])$ vertically into the region above $S$ with fixing $\Gamma$ since $(M_0,\Gamma\cap M_0)$ has a product structure by the conditions (6), where $\epsilon$ is a sufficiently small positive real number.
\end{proof}

\subsection{A sufficient condition for spatial graph to have the representativity at least $n$}

The following lemma is useful to show that $r(F,\Gamma)\ge n$.

\begin{lemma}\label{n-connected}
Let $F$ be a closed surface of positive genus which separates $S^3$ into $V_1$ and $V_2$, and $\Gamma$ be a spatial graph which is contained in $F$.
Suppose that there exists a collection of essential disks $\mathcal{D}_i=\{D_{\lambda}^i\}$ in $V_i$ for $i=1,2$ whose boundary intersects $\Gamma$ transversely in the interior of edges such that for a collection $\mathcal{F}_i=\{F_{\mu}^i\}$ of the components of $cl(F-\bigcup_{\lambda} N(\partial D_{\lambda}^i))$,
\begin{enumerate}
\item for any essential loop $l$ in $F_{\mu}^i$, $|l\cap \Gamma|\ge n$, and
\item for any essential arc $a$ in $F_{\mu}^i$ such that $\partial a$ is contained in a single component of $\partial F_{\mu}^i$, $|a\cap \Gamma|\ge n/2$.
\end{enumerate}
Then, $r(F,\Gamma)\ge n$.
\end{lemma}

\begin{proof}
Suppose that $r(F,\Gamma)< n$ and without loss of generality there exists a compressing disk $D$ in $V_1$ for $F$ such that $|\partial D\cap \Gamma|< n$.
We assume that $|D\cap \bigcup_{\lambda}D_{\lambda}^1|$ is minimal.


If $D\cap \bigcup_{\lambda} D_{\lambda}^1=\emptyset$, then $\partial D$ is an essential loop in a component $F_{\mu}^1$ with $|\partial D\cap \Gamma|<n$.
This contradicts the condition (1).

Otherwise, let $\alpha$ be an arc of $D\cap \bigcup_{\lambda}D_{\lambda}^1$ which is outermost in $D$, $\delta$ be the corresponding outermost disk in $D$ bounded by $\alpha$, and put $\beta=\partial \delta-\text{int}\alpha$.
We may take $\alpha$ and $\delta$ so that $|\beta\cap \Gamma|< n/2$ since $|\partial D\cap \Gamma|< n$.
This contradicts the condition (2).
\end{proof}

\subsection{Constructing knots with the representativity exactly $n$}

\begin{lemma}\label{exactly}
For any Heegaard surface $F$ of genus $g\ge 1$ and for any integer $n\ge 2$, there exists a knot $K$ non-separatively contained in $F$ such that $r(F,K)=n$.
\end{lemma}

\begin{proof}
Let $F$ be a genus $g$ Heegaard surface, $m_0,m_1,\ldots,m_g$ be meridian loops and $l_0,l_1,\ldots,l_g$ be longitude loops on $F$ as Figure \ref{knot}.
We take $n+1$ parallel copies of $m_0$, $n$ parallel copies of $m_1$, $\lceil n/2 \rceil$ parallel copies of $m_i$ $(i=2,\ldots,g)$, $\lceil n/2 \rceil$ parallel copies of $l_0$, $\lfloor n/2 \rfloor$ parallel copies of $l_1$, $\lceil n/2 \rceil$ parallel copies of $l_j$ $(j=2,\ldots,g)$, where $\lceil x \rceil$ denotes the ceiling function of $x$ which is the smallest integer not less than $x$, and $\lfloor x \rfloor$ denotes the floor function of $x$ which is the largest integer not greater than $x$.
Then a knot $K$ is obtained from these parallel loops by smoothing the intersection in a same direction (which is denoted by $+$).
See Figure \ref{non-torus_fig} for a link $L=7m_0+7m_1+7m_2 + 2l_0+2l_1+2l_2$ on a genus $2$ Heegaard surface.
Thus we have
\begin{eqnarray*}
K & = & (n+1)m_0+nm_1+\lceil n/2 \rceil m_2+\cdots+\lceil n/2 \rceil m_g\\
& & + \lceil n/2 \rceil l_0 + \lfloor n/2 \rfloor l_1 + \lceil n/2 \rceil l_2 +\cdots + \lceil n/2 \rceil l_g
\end{eqnarray*}

\begin{figure}[htbp]
	\begin{center}
	\includegraphics[trim=0mm 0mm 0mm 0mm, width=.8\linewidth]{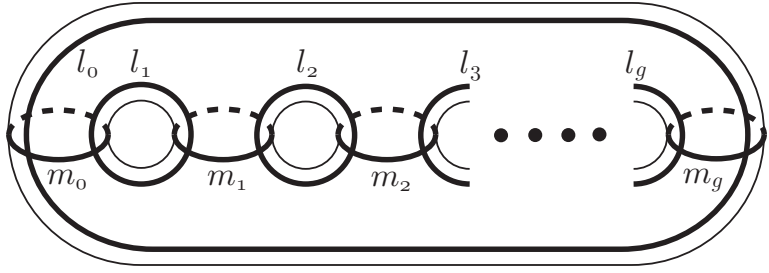}
	\end{center}
	\caption{meridian loops $m_i$ and longitude loops $l_j$ $(i,j=1,\ldots,g)$}
	\label{knot}
\end{figure}

We use Lemma \ref{n-connected} as follows.
Let $D_i^1$ $(i=0,\ldots,g)$ be an essential disk which is bounded by $m_i$ in a genus $g$ handlebody $V_1$ bounded by $F$, and $D_j^2$ $(j=0,\ldots,g)$ be an essential disk which is bounded by $l_j$ in a genus $g$ handlebody $V_2$ bounded by $F$.
By isotoping $D_i^1$ $(i=0,\ldots,g)$ and $D_j^2$ $(j=0,\ldots,g)$ slightly, we may assume that $\partial D_i^1$ $(i=0,1)$ intersects $K$ in $\lceil n/2 \rceil + \lfloor n/2 \rfloor =n$ points, $\partial D_i^1$ $(i=2,\ldots,g)$ intersects $K$ in $2 \lceil n/2 \rceil$ $(\ge n)$ points, $\partial D_0^2$ intersects $K$ in $n+1+\lceil n/2 \rceil$ points, $\partial D_1^2$ intersects $K$ in $2n+1$ points, $\partial D_2^2$ intersects $K$ in $n+\lceil n/2 \rceil$ points, $\partial D_j^2$ $(j=3,\ldots,g)$ intersects $K$ in $2\lceil n/2 \rceil$ points.

If we cut $F$ along $\partial D_i^1$ $(i=0,\ldots,g)$ and $\partial D_j^2$ $(j=0,\ldots,g)$, then we obtain $\mathcal{F}_1=\{F_+^1,F_-^1\}$ and $\mathcal{F}_2=\{F_+^2,F_-^2\}$ respectively each of which consists of two planar surfaces with $g+1$ boundary components.
By the construction, $K\cap F_{\pm}^1$ consists of $\lceil n/2 \rceil$ parallel arcs connecting $\partial D_0^1$ and $\partial D_g^1$, $\lfloor n/2 \rfloor$ parallel arcs connecting $\partial D_0^1$ and $\partial D_1^1$, $\lceil n/2 \rceil$ parallel arcs connecting $\partial D_i^1$ and $\partial D_{i+1}^1$ $(i=1,\ldots,g-1)$, and $K\cap F_{\pm}^2$ consists of $\lceil n/2 \rceil$ parallel arcs connecting $\partial D_0^2$ and $\partial D_g^2$, $n+1$ parallel arcs connecting $\partial D_0^2$ and $\partial D_1^2$, $n$ parallel arcs connecting $\partial D_1^2$ and $\partial D_2^2$, $\lceil n/2 \rceil$ parallel arcs connecting $\partial D_j^2$ and $\partial D_{j+1}^2$ $(j=2,\ldots,g-1)$.

Therefore, both conditions in Lemma \ref{n-connected} are satisfied for $\mathcal{F}_1=\{F_+^1,F_-^1\}$, namely (1) for any essential loop $l$ in $F_{\pm}^1$, $|l\cap K|\ge n$, and (2) for any essential arc $a$ in $F_{\pm}^1$ such that $\partial a$ is contained in a single component of $\partial F_{\pm}^1$, $|a\cap K|\ge n/2$.
And also, both conditions in Lemma \ref{n-connected} are satisfied for $\mathcal{F}_2=\{F_+^2,F_-^2\}$, namely (1) for any essential loop $l$ in $F_{\pm}^2$, $|l\cap K|\ge 2 \lceil n/2 \rceil$, and (2) for any essential arc $a$ in $F_{\pm}^2$ such that $\partial a$ is contained in a single component of $\partial F_{\pm}^2$, $|a\cap K|\ge \lceil n/2 \rceil$.

Hence by Lemma \ref{n-connected}, $r(F,K)\ge n$.
Moreover, since $\partial D_0^1$ intersects $K$ in $n$ points, we have $r(F,K)=n$.
Finally, since $\partial D_0^1$ intersects $K$ in $n$ points and $\partial D_1^2$ intersects $K$ in $2n+1$ points, $K$ is a non-separating loop in $F$.
\end{proof}

We need the following two theorems due to Fox (also Scharlemann--Thompson) and Bonahon to prove Theorem \ref{arbitrarily}.

\begin{theorem}[\cite{F}, {\cite[Theorem 7]{ST}}]\label{Fox}
A connected compact 3-dimensional submanifold of $S^3$ can be reimbedded in $S^3$ so that it is the complement of a union of handlebodies in $S^3$.
\end{theorem}

We review on a characteristic compression body from \cite[page 243]{B}.
Let $M$ be an irreducible compact 3-manifold with boundary and let $\mathcal{D}\subset M$ be a collection of disjoint compression disks for $\partial M$.
If $V$ is the union of a regular neighbourhood $U$ of $\mathcal{D}\cup\partial M$ and of all the components of the closure $cl(M-U)$ that are 3-balls, then $V$ is a compression body for $\partial M$, where $\partial M\subset \partial V$ is the {\em exterior boundary} of $V$ which is denoted by $\partial_e V$, and $\partial V-\partial_e V$ is the {\em interior boundary} of $V$ which is denoted by $\partial_i V$.
A {\em characteristic compression body} $V$ of $M$ is a compression body constructed by the above procedures such that $M-\text{int}V$ is $\partial$-irreducible, namely, $\partial_i V$ is incompressible in $M-\text{int}V$.

\begin{theorem}[{\cite[Theorem 2.1]{B}}]\label{Bonahon}
An irreducible compact 3-manifold with boundary has a unique characteristic compression body.
\end{theorem}

\begin{lemma}\label{non-separating}
Let $G$ be a graph with $g(G)\ge 1$ which is contained in a closed surface $F$ of genus $g(G)$.
Then there exists a cycle $C$ of $G$ which is non-separating in $F$.
\end{lemma}

\begin{proof}
We take a spanning tree $T$ of $G$.
By contracting $T$, we obtain a bouquet $G/T$ in $F$.
We remark that any component of $F-(G/T)$ is an open disk since the genus of $F$ coincides with $g(G)$.
Suppose that any loop of $G/T$ is separating in $F$ and let $l$ be an outermost essential loop of $G/T$ in $F$.
Then $l$ cuts off a once puctured surface of genus $g\ge 1$ which contains a non-open disk region of $F-(G/T)$, a contradiction.
Hence $G/T$ has at least one non-separating loop and the corresponding cycle $C$ of $G$ is non-separating in $F$.
\end{proof}

\subsection{Primitive spatial graphs}

We say that a spatial graph $\Gamma$ is {\em free} if the fundamental group $\pi_1(S^3-\Gamma)$ is free.


A spatial graph $\Gamma$ is said to be {\it primitive} if for each component $\Gamma_i$ of $\Gamma$ and any spanning tree $T_i$ of $\Gamma_i$, the bouquet $\Gamma_i/T_i$ obtained from $\Gamma_i$ by contracting all edges of $T_i$ is trivial.
In \cite{OT2}, we showed that a spatial graph $\Gamma$ is primitive if and only if for any connected spatial subgraph $H\subset \Gamma$, $H$ is free.

We say that a spatial graph $\Gamma$ is {\em minimally knotted} ({\em almost trivial}) if any proper spatial subgraph of $\Gamma$ is trivial, but $\Gamma$ itself is not trivial.




\begin{lemma}\label{primitive}
If $\Gamma$ is primitive, then $r(\Gamma)\le \beta_1(G)$.
\end{lemma}

\begin{proof}
If $\Gamma$ is primitive, then $(E(T),\Gamma\cap E(T))$ is a trivial $\beta_1(G)$-string tangle for any spanning tree $T$ of $\Gamma$, where $E(T)=S^3-\text{int}N(T)$.
Thus $(N(T), \Gamma\cap N(T))\cup (E(T),\Gamma\cap E(T))$ is a bridge tangle decomposition for $(S^3,\Gamma)$ and $\partial N(T)$ is a bridge tangle decomposing sphere for $\Gamma$ such that $|\Gamma\cap \partial N(T)|=2\beta_1(G)$.
Hence by Theorem \ref{main2}, $r(\Gamma)\le \beta_1(G)$.
\end{proof}

\begin{lemma}\label{not primitive}
If $\Gamma$ is not primitive, then $\Gamma$ contains a connected totally knotted spatial subgraph.
\end{lemma}

\begin{proof}
If $\Gamma$ is not primitive, then there exists a spanning tree $T$ of $\Gamma$ such that $\Gamma/T$ is a non-trivial bouquet.
Then by taking a minimal non-trivial subgraph $\Gamma_0/T_0$ of $\Gamma/T$, we have a minimally knotted subgraph $\Gamma_0/T_0\subset \Gamma/T$.
Since a minimally knotted spatial graph is totally knotted (\cite{OT}), $\Gamma_0/T_0$ is totally knotted.
Then the corresponding spatial subgraph $\Gamma_0$ is a connected totally knotted spatial subgraph of $\Gamma$.
\end{proof}

\section{Proofs of Theorems}


\begin{proof} (Theorem \ref{main})
Let $F$ be a 2-sphere which contains $\Gamma$.
By Lemma \ref{lemma1} and \ref{lemma2}, we may assume that a pair $(F,\Gamma)$ is in a Morse bridge position and $F$ has only two critical points.
This shows that $F$ intersects $S^2$ in a single loop.
\end{proof}

\begin{proof} (Theorem \ref{main2})
Let $(B_+,\Gamma_+)\cup_{S^2}(B_-,\Gamma_-)$ be a bridge tangle decomposition of $\Gamma$ such that $|\Gamma\cap S^2|=bs(\Gamma)$.
Let $F$ be any closed surface of positive genus containing $\Gamma$.
By Lemma \ref{lemma1} and \ref{lemma2}, we may assume that a pair $(F,\Gamma)$ is in a Morse bridge position and $F$ has no inessential saddle point.
Since $F$ has a positive genus, there exists an essential saddle point $p$.
This shows that there exists a regular value $t\in [c_-,c_+]$ such that $h^{-1}(t)\cap F$ contains at least two essential loops in $F$.

Let $C$ be a loop of $h^{-1}(t)\cap F$ which is essential in $F$ and innermost in $h^{-1}(t)$.
Since $h^{-1}(t)$ is a 2-sphere, we can take this loop $C$ so that $|C\cap \Gamma|\le |\Gamma \cap h^{-1}(t)|/2$.
Let $D$ be the corresponding innermost disk in $h^{-1}(t)$ bounded by $C$.
Since $C$ is innermost in $h^{-1}(t)$, any loop of $\text{int}D\cap F$ is inessential in $F$.

Let $\alpha$ be a loop of $\text{int}D\cap F$ which is innermost in $F$, and $\delta$ be the disk in $F$ bounded by $\alpha$.
Then by cutting and pasting $D$ along $\delta$, we have a new disk $D'$ such that $|\text{int}D'\cap F|<|\text{int}D\cap F|$.
Eventually we have $\text{int}D\cap F=\emptyset$, then $D$ is a compressing disk for $F$ which shows that
\[
\displaystyle r(F,\Gamma)\le |\Gamma\cap \partial D|\le \frac{|\Gamma\cap S^2|}{2}= \frac{bs(\Gamma)}{2}.
\]
\end{proof}

\begin{proof} (of Theorem \ref{arbitrarily})
Let $F$ be a closed surface in $S^3$ with $g(F)\ge g(G)$, and put $S^3=M_1\cup_F M_2$.

First by using Theorem \ref{Bonahon}, we take a characteristic compression body ({\em i.e.} maximal compression body) $V_i$ for $F$ in $M_i$ $(i=1,2)$.
Thus $\partial_e V_i$ coincides with $F$ and $\partial_i V_i$ is incompressible in $M_i-\text{int}V_i$, where $\partial_e V_i$ denotes the exterior boundary of $V_i$ and $\partial_i V_i$ denotes the interior boundary of $V_i$.

Next by using Theorem \ref{Fox}, we reimbed $V_1\cup_F V_2$ into $S^3$ so that the image $V_1'\cup_{F'} V_2'$ is the closed complement of a union of handlebodies embedded in $S^3$.
Since $V_i'\cup_{\partial_i V_i'}$(a union of handlebodies), which is denoted by $M_i'$, is a handlebody, $F$ becomes a Heegaard surface $F'$ in the new $S^3=M_1'\cup_{F'}M_2'$.

By Lemma \ref{exactly}, there exists a knot $K'$ in $F'$ which is non-separating in $F'$ and satisfies $r(F',K')=n$ for a given integer $n$.

We construct a spatial graph $\Gamma'$ of $G$ contained in $F'$ such that $\Gamma'$ contains $K'$ as a non-separating cycle as follows.
Let $G$ be a non-planar graph embedded in a closed orientable surface $F_0$ with $g(F_0)=g(G)$.
Then, by Lemma \ref{non-separating}, there exists a cycle $C$ of $G$ which is non-separating in $F$.
We take a connected sum of $F_0$ and $F_1$, where $F_1$ is a closed orientable surface of genus $g(F')-g(F_0)$.
Let $\phi:F_0\# F_1\to F'$ be a homeomorphism such that $\phi(C)=K'$, and put $\Gamma'=\phi(G)$.
Then we obtain a desired triple $F'\supset \Gamma'\supset K'$, where $K'$ is also a cycle of $\Gamma'$.

\begin{center}
$S^3=M_1\cup_F M_2\supset V_1\cup_F V_2\supset F\supset \Gamma\supset K$

$\downarrow$ Fox's reimbedding

$S^3=M_1'\cup_{F'} M_2'\supset V_1\cup_{F'} V_2\supset F'\supset \Gamma'\supset K'$
\end{center}

Finally we restore the Fox's reimbedding and obtain a spatial graph $\Gamma$ and a knot $K$ as the preimage.
Then we have that $r(F,\Gamma)\ge r(F',\Gamma')\ge r(F',K')=n$.
\end{proof}

\begin{proof} (of Theorem \ref{totally knotted})
This follows from Lemmas \ref{primitive} and \ref{not primitive}.
\end{proof}

\begin{proof} (of Theorem \ref{spatially})
Let $F$ be a closed surface containing $\Gamma$ such that $r(F,\Gamma)=n$.
Suppose that there exists an essential tangle decomposing sphere $S$ for $\Gamma$ with $|\Gamma\cap S|<n$.
We may assume that $F$ intersects $S$ in loops, and assume that $|F\cap S|$ is minimal.
Then the innermost disk $D$ in $S$ bounded by an innermost loop of $F\cap S$ is a compressing disk for $F$ since $S$ is an essential tangle decomposing sphere.
It follows that $r(F,\Gamma)\le |\partial D\cap \Gamma|\le |\Gamma\cap S|<n$.
\end{proof}

\section{Applications to examples}\label{determination}

\subsection{Minimally knotted handcuff graphs and theta curves}

\begin{example}
The left side of Figure \ref{3-ball} shows a pair $(F,\Gamma)$ of a genus two Heegaard surface $F$ and a minimally knotted (hence totally knotted by \cite{OT}) handcuff graph $\Gamma$ which is in a Morse bridge position.
To show $r(F,\Gamma)\ge 2$, let $D_1^i$ and $D_2^i$ be meridian disks in a handlebody $V_i$ as in Figure \ref{3-ball} and then, following Lemma \ref{n-connected}, we have a $4$-punctured sphere $F_1^i$.
Since (1) for any essential loop $l$ in $F_1^i$, $|l\cap \Gamma|\ge 2$ and (2) for any essential arc $a$ in $F_1^i$ such that $\partial a$ is contained in a single component of $\partial F_1^i$, $|a\cap \Gamma|\ge 1$, by Lemma \ref{n-connected}, $r(F,\Gamma)\ge 2$.
On the other hand, an inequality $r(\Gamma)\le bs(\Gamma)/2\le 5/2$ holds by Theorem \ref{main2}.
Hence we have $r(\Gamma)=2$.


\begin{figure}[htbp]
	\begin{center}
	\begin{tabular}{cc}
	\includegraphics[trim=0mm 0mm 0mm 0mm, width=.4\linewidth]{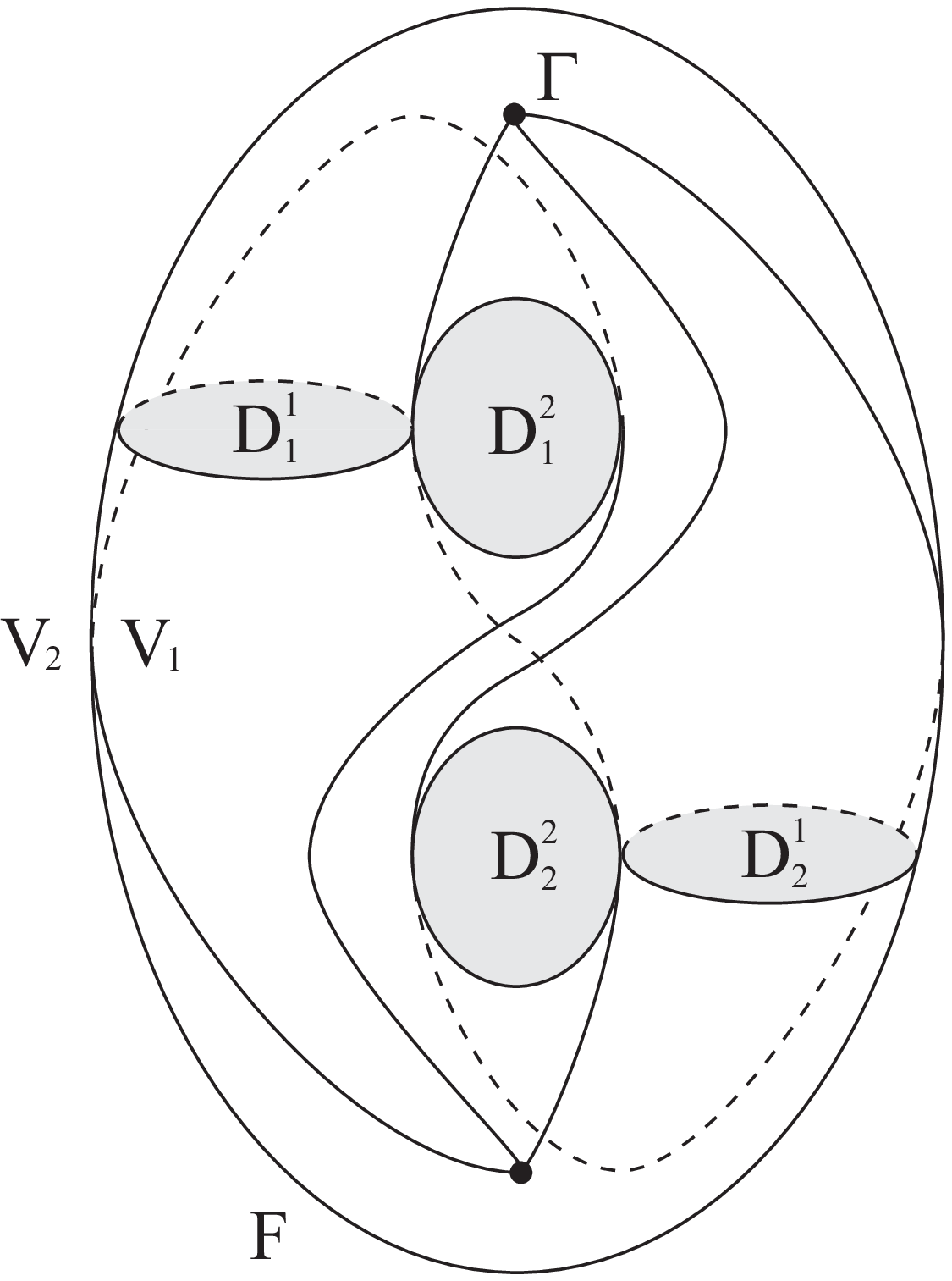}&
	\includegraphics[trim=0mm 0mm 0mm 0mm, width=.35\linewidth]{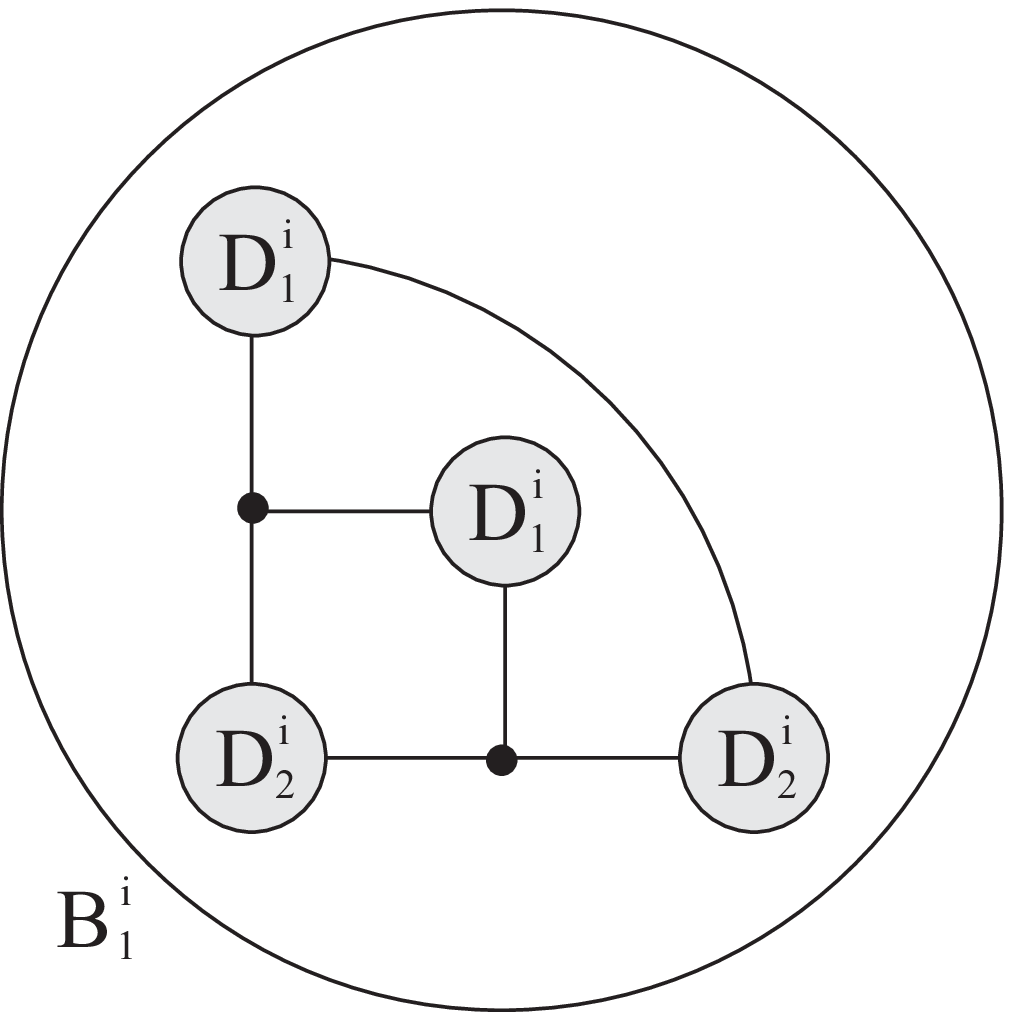}\\
	$\mathcal{D}_i$ for $(F,\Gamma)$ & $\Gamma\cap F_1^i$ on $F_1^i$
	\end{tabular}
	\end{center}
	\caption{The simplest minimally knotted handcuff graph}
	\label{3-ball}
\end{figure}

\end{example}

\begin{example}
Figure \ref{theta} shows a pair $(F,\Gamma)$ of a genus three or two closed surface $F$ and a minimally knotted Kinoshita's theta curve (\cite{K}).
By using Lemma \ref{n-connected}, $r(F,\Gamma)\ge 2$, and on the other hand, by Theorem \ref{main2}, $r(\Gamma)\le bs(\Gamma)/2\le 5/2$.
Hence we have $r(\Gamma)=2$.

\begin{figure}[htbp]
	\begin{center}
	\begin{tabular}{cc}
	\includegraphics[trim=0mm 0mm 0mm 0mm, width=.4\linewidth]{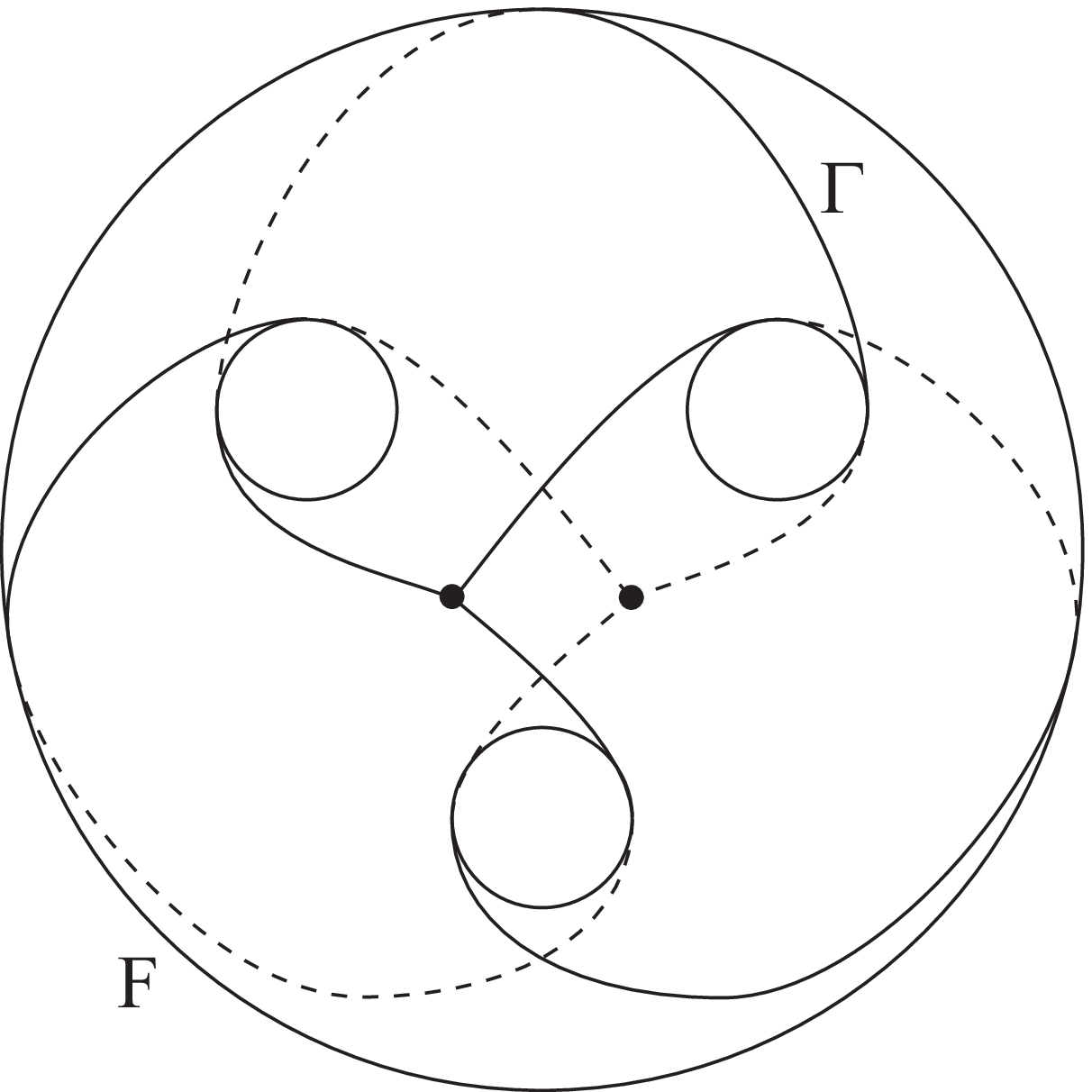}&
	\includegraphics[trim=0mm 0mm 0mm 0mm, width=.4\linewidth]{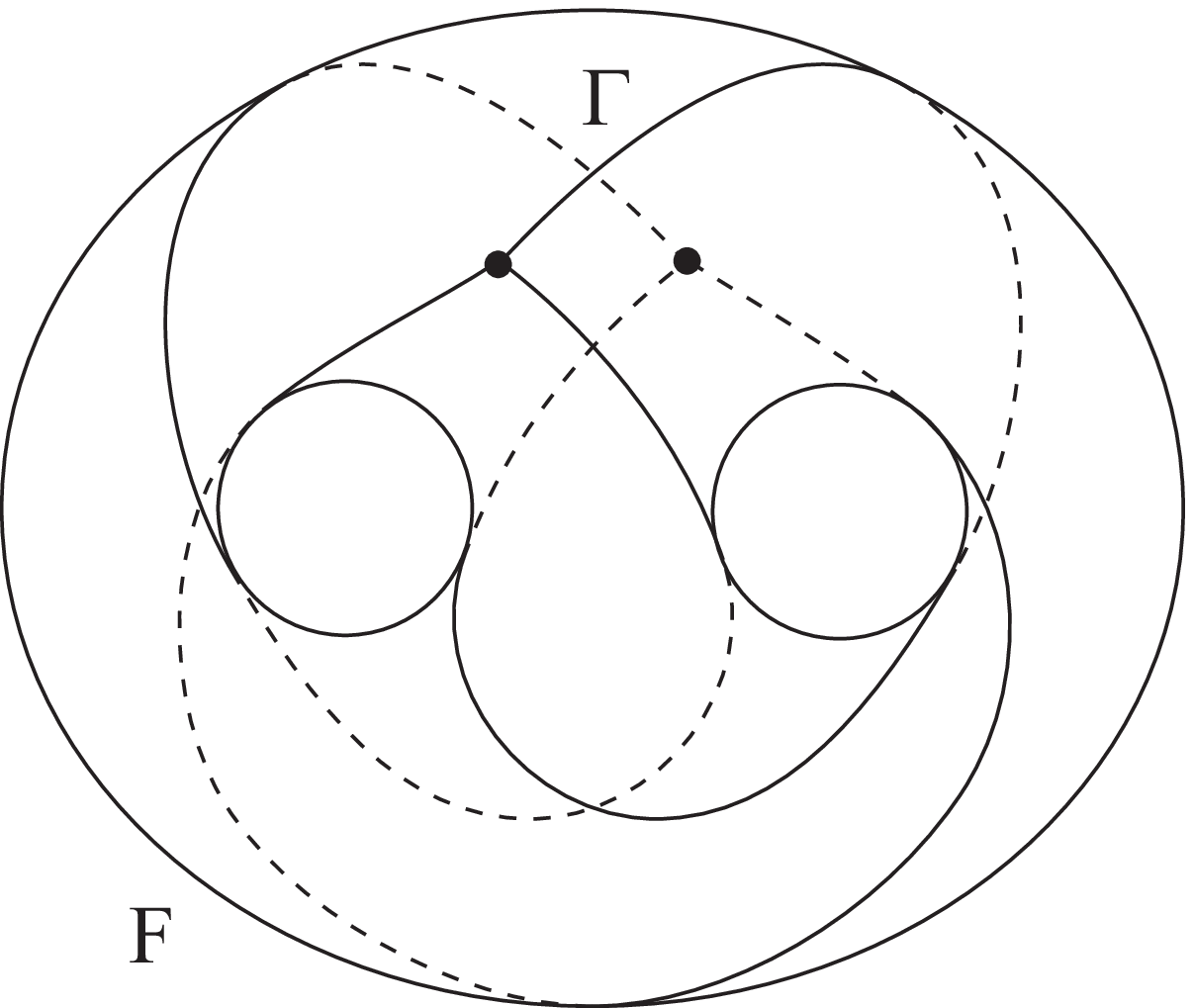}\\

	\end{tabular}
	\end{center}
	\caption{Kinoshita's theta curve on two closed surfaces}
	\label{theta}
\end{figure}

\end{example}

\subsection{Determination of the bridge number of some non-torus links}

The following theorem says that by using the representativity, we can determine the bridge string number of the link.
We use Theorem \ref{main2}, Lemmas \ref{lemma1}, \ref{lemma2} and \ref{n-connected}.

\begin{theorem}\label{non-torus}
For any positive even integer $n$, there exists a non-torus link $L$ such that $r(L)\ge n$ and $bs(L)=3n$.
\end{theorem}

\begin{proof}
We use the notation defined in Lemma \ref{exactly}.
Let $S^3=V_1\cup_F V_2$ be a genus two Heegaard splitting and $\mathcal{D}_i=\{D_{\lambda}^i \}$ $(i=1,2,\ \lambda=0,1,2)$ be a collection of meridian disks in $V_i$, where $D_{\lambda}^1$ is bounded by $m_{\lambda}$ and $D_{\lambda}^2$ is bounded by $l_{\lambda}$.


We take $p\ge 1$ parallel copies of $\partial D_{\lambda}^2$ and $q>3p$ parallel copies of $\partial D_{\lambda}^1$.
Let $L(p,q)$ be a link in $F$ which is obtained from these loops by smoothing the intersection in a same direction, namely,
\begin{eqnarray*}
L(p,q) & = & qm_0+qm_1+qm_2 + p l_0 + p l_1 + p l_2.
\end{eqnarray*}
See Figure \ref{non-torus_fig} for an example of $L(2,7)$.

\begin{figure}[htbp]
	\begin{center}
	\begin{tabular}{ccc}
	\includegraphics[trim=0mm 0mm 0mm 0mm, width=.3\linewidth]{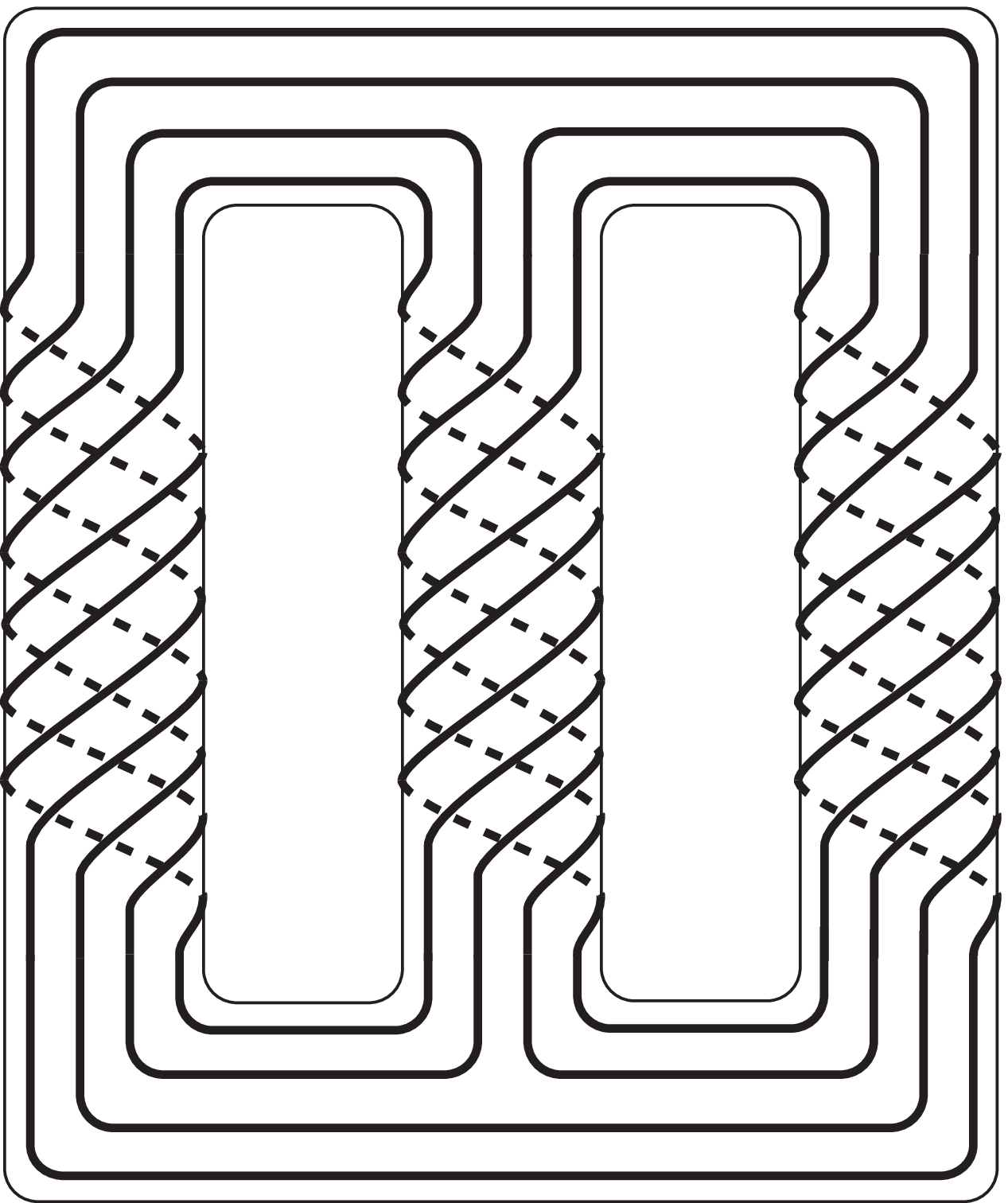}&
	\includegraphics[trim=0mm 0mm 0mm 0mm, width=.3\linewidth]{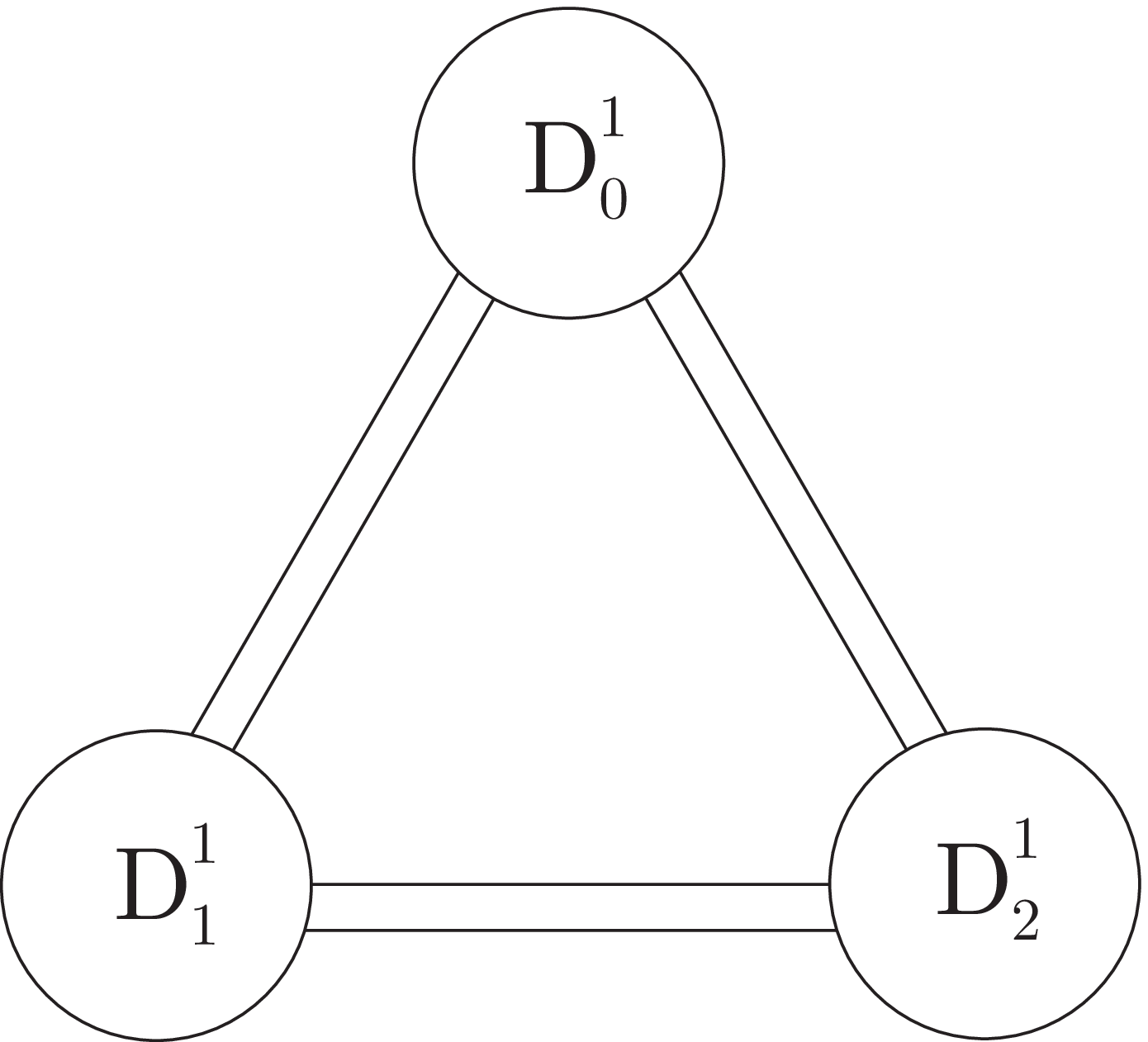}&
	\includegraphics[trim=0mm 0mm 0mm 0mm, width=.3\linewidth]{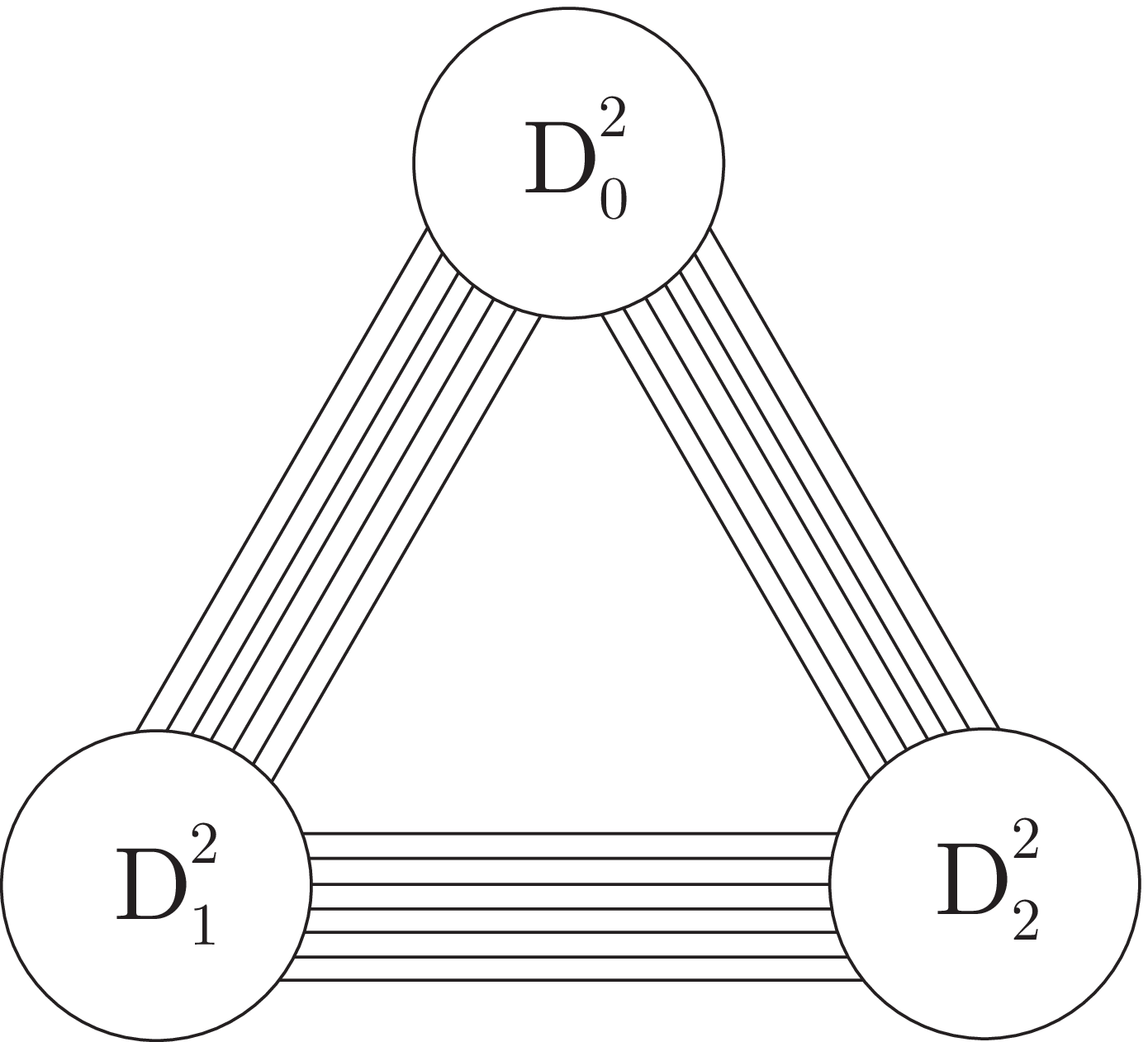}\\
	$L(2,7)$ & $F_{\mu}^1$ & $F_{\mu}^2$\\
	\end{tabular}
	\end{center}
	\caption{Non-torus link $L(2,7)=7m_0+7m_1+7m_2 + 2l_0+2l_1+2l_2$}
	\label{non-torus_fig}
\end{figure}

Put $n=2p$.
Following Lemma \ref{n-connected}, let $\mathcal{F}_i=\{F_{\mu}^i\}$ be a collection of the components of $cl(F-\bigcup_{\lambda} N(\partial D_{\lambda}^i))$.
Then $F_{\mu}^i$ is a pair of pants, $L \cap F_{\mu}^1$ consists of three classes of $p$ parallel arcs which connect two distinct components of $\partial F_{\mu}^1$, and $L \cap F_{\mu}^2$ consists of three classes of $q$ parallel arcs which connect two distinct components of $\partial F_{\mu}^2$.
Therefore, it holds that (1) for any essential loop $l$ in $F_{\mu}^1$ (resp. $F_{\mu}^2$), $|l\cap L|\ge 2p$ (resp. $|l\cap L|\ge 2q$), (2) for any essential arc $a$ in $F_{\mu}^1$ (resp. $F_{\mu}^2$) such that $\partial a$ is contained in a single component of $\partial F_{\mu}^1$ (resp. $\partial F_{\mu}^2$), $|a\cap L|\ge p$ (resp. $|a\cap L|\ge q$).
Hence, by Lemma \ref{n-connected}, $r(F,L)\ge 2p=n$ in $V_1$ and $r(F,L)\ge 2q=3n$ in $V_2$, thus we have $r(F,L)\ge n$.
(In fact, $r(F,L)=n$.)


We see $bs(L)\le 6p$ from Figure \ref{non-torus_fig}.
To show $bs(L)=3n$, suppose that $bs(L)<6p=3n$ and $(S^3,L)=(B_+,L_+)\cup_S(B_-,L_-)$ be a bridge tangle decomposition with $bs(L)<3n$.
We apply Lemma \ref{lemma2} to a pair $(F,L)$ of the genus two Heegaard surface $F$ and the link $L$ with the bridge tangle decomposition $(B_+,L_+)\cup_S(B_-,L_-)$.
Then $(F,L)$ is in a Morse bridge position with respect to the standard height function $h$ and $F$ has no inessential saddle point.
If there exists a level sphere $h^{-1}(t)$ such that $h^{-1}(t)\cap F$ contains three or more parallel classes of loops which are essential in $F$, then $r(F,L)\le bs(L)/3<3n/3=n$ in the same way as the proof of Theorem \ref{main2}.
However this contradicts $r(F,L)\ge n$.
Hence we have the condition: 

(*) {\em for each regular value $t$, $h^{-1}(t)\cap F$ contains at most two parallel classes of loops which are essential in $F$.}

In the following, we review types of saddle points which are defined in \cite{MO2}.
Let $p$ be a saddle point of $F$ which corresponds the critical value $t_p\in \Bbb{R}$.
Let $X_p$ be a pair of pants component of $F\cap h^{-1}([t_p-\epsilon, t_p+\epsilon])$ containing $p$ for a fixed sufficiently small positive real number $\epsilon$.
Let $C_1$, $C_2$ and $C_3$ be the boundary components of $X_p$, where we assume that $C_1$ and $C_2$ are contained in the same level $h^{-1}(t_p\pm\epsilon)$, and $C_3$ is contained in the another level $h^{-1}(t_p\mp\epsilon)$.
See Figure \ref{saddle_fig}.
We call a saddle point $p$
\begin{enumerate}
\item {\em Type I} if all of $C_1$, $C_2$ and $C_3$ are inessential in $F$,
\item {\em Type II} if exactly one of $C_1$ and $C_2$ is essential and $C_3$ is essential in $F$,
\item {\em Type III} if both of $C_1$ and $C_2$ are essential and $C_3$ is inessential in $F$,
\item {\em Type IV} if all of $C_1$, $C_2$ and $C_3$ are essential in $F$.
\end{enumerate}

\begin{figure}[htbp]
	\begin{center}
	\includegraphics[trim=0mm 0mm 0mm 0mm, width=.5\linewidth]{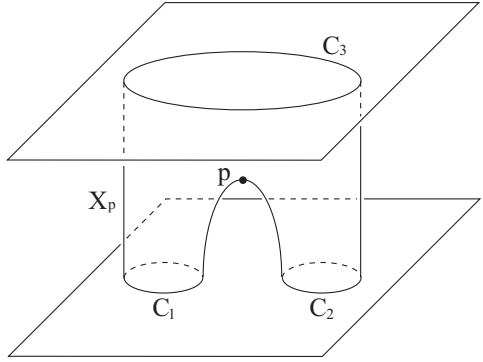}
	\end{center}
	\caption{a pair of pants component $X_p$ and three boundary components $C_1$, $C_2$ and $C_3$}
	\label{saddle_fig}
\end{figure}

We note that saddle points of Type I and II are inessential, and saddle points of Type III and IV are essential.
By \cite[Lemma 2.3 (3)]{MO2}, there exists a saddle point $p$ of Type IV since the genus of $F$ is greater than one.
Then each essential loop $C_i$ bounds a disk $D_i$ in $h^{-1}(t_p\pm\epsilon)$ so that $D_1\cup D_2\cup D_3\cup X_p$ forms a 2-sphere which is isotopic to a level 2-sphere $h^{-1}(t_p)$ in $h^{-1}([t_p-\epsilon, t_p+\epsilon])$, where $t_p$ is the critical value corresponding to $p$ and $X_p$ is a pair of pants component of $F\cap h^{-1}([t_p-\epsilon, t_p+\epsilon])$ containing $p$.
See Figure \ref{2-sphere}.

\begin{figure}[htbp]
	\begin{center}
	\includegraphics[trim=0mm 0mm 0mm 0mm, width=.5\linewidth]{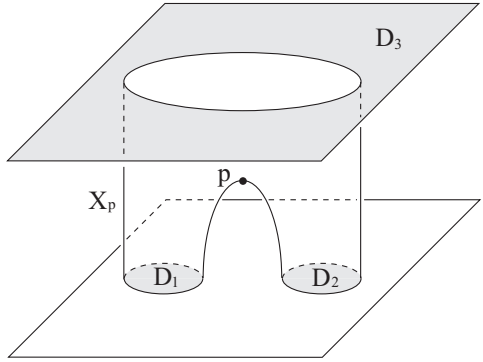}
	\end{center}
	\caption{$D_1\cup D_2\cup D_3\cup X_p$ forms a 2-sphere}
	\label{2-sphere}
\end{figure}

\begin{claim}
$F\cap \text{int}D_3$ does not contain a loop which is essential in $F$.
\end{claim}

\begin{proof}
If $F\cap \text{int}D_3$ contains a loop which is essential in $F$, then $h^{-1}(t_p\pm\epsilon)-(D_1\cup D_2)$ also contains a loop which is essential in $F$ and not parallel to neither $C_1$ nor $C_2$.
This contradicts the condition (*).
\end{proof}

\begin{claim}
At least one of $F\cap \text{int}D_1$ and $F\cap \text{int}D_2$ does not contain a loop which is essential in $F$.
\end{claim}

\begin{proof}
If both of $F\cap \text{int}D_1$ and $F\cap \text{int}D_2$ contain a loop which is essential in $F$, then $h^{-1}(t_p\mp\epsilon)-D_3$ contains mutually non-parallel two loops which are essential in $F$ and not parallel to $C_3$.
This contradicts the condition (*).
\end{proof}

Without loss of generality, we assume that $F\cap \text{int}D_2$ does not contain a loop which is essential in $F$.
In the same way as the proof of Theorem \ref{main2}, by cutting and pasting $D_3$ and $D_2$, we have new disks $D_3'$ and $D_2'$ which are bounded by $C_3$ and $C_2$ respectively, and $\text{int}D_3'\cap F=\emptyset$ and $\text{int}D_2'\cap F=\emptyset$.

Then it follows that (a) $D_3\subset V_1$ and $D_2\subset V_2$ or (b) $D_3\subset V_2$ and $D_2\subset V_1$.
In case (a), we have
\[
|\partial D_2\cap L|\le bs(L)<3n=6p<2q.
\]
However this contradicts that $r(F,L)\ge 2q$ in the side of $V_2$.
In case (b), similarly we have
\[
|\partial D_3\cap L|\le bs(L)<3n=6p<2q.
\]
However this also contradicts that $r(F,L)\ge 2q$ in the side of $V_2$.
Therefore we have $bs(L)=3n$.
%
\end{proof}

\section{Related problems}

\subsection{Waist and the representativity of knots}

In \cite{MO2}, the {\em waist} of a non-trivial knot $K$ is defined as
\[
waist(K)=\max_{F\in\mathcal{F}} \min_{D\in\mathcal{D}_F} |D\cap K|,
\]
where $\mathcal{F}$ is the set of all closed incompressible surfaces in $S^3-K$ and $\mathcal{D}_F$ is the set of all compressing disks for $F$ in $S^3$.
We define that $waist(K)=0$ for the trivial knot $K$.
Then it has been shown in \cite{MO2} that
\[
\displaystyle waist(K)\le \frac{bs(K)}{3}.
\]
With Theorem \ref{main2}, the following question occurred.

\begin{problem}
Does it hold that $waist(K)\le r(K)$ for a knot $K$?
\end{problem}

\subsection{Closed genus of knots}

Finally, we introduce a new numerical invariant for knots.
By Theorem \ref{knots}, we know that $r(K)\ge 2$ for any non-trivial knot $K$, and hence we can define the {\em closed genus} of a non-trivial knot $K$ as
\[
cg(K)=\min_{F\in \mathcal{F}} \{g(F)|F\supset K,\ r(F,K)\ge 2\},
\]
where $\mathcal{F}$ is the set of closed surfaces.
Then $cg(K)=1$ if and only if $K$ is a torus knot or a cable knot.

Generally, it holds that $cg(K)\le 2g(K)$ for any non-trivial knot $K$.
If $cg(K)<2g(K)$ holds, then $K$ satisfies the Neuwirth conjecture (\cite{N}): for any non-trivial knot $K$, there exists a closed surface $S$ containing $K$ non-separatively such that $S\cap E(K)$ is essential in $E(K)$.



\begin{problem}
Does a closed surface $F$ which gives $cg(K)$ give also $r(K)$?
\end{problem}

\bigskip


\bibliographystyle{amsplain}

\end{document}